\documentclass[12pt,leqno]{article}
\usepackage{amsfonts}
\usepackage{amsmath, amsthm, amsfonts, amssymb, color}
\usepackage{mathrsfs}
\usepackage{color}
\newtheorem{thm}{Theorem}[section]

\newtheorem{lem}[thm]{Lemma}

\theoremstyle{definition}

\newcommand{\scr}[1]{\mathscr #1}
\definecolor{wco}{rgb}{0.5,0.2,0.3}

\numberwithin{equation}{section} \theoremstyle{remark}

\newcommand{\ua}{\uparrow}

\title{{\bf  Distribution Dependent SDEs with Singular Coefficients}\footnote{Supported in
 part by  NNSFC (11771326, 11431014).} }
\author{
{\bf   Xing Huang $^{a)}$,  Feng-Yu Wang $^{a), b)}$  }\\
\footnotesize{ a)Center for Applied Mathematics, Tianjin
University, Tianjin 300072, China}\\
\footnotesize{  xinghuang@tju.edu.cn}\\
 \footnotesize{ b)Department of Mathematics,
Swansea University, Singleton Park, SA2 8PP, United Kingdom}\\
\footnotesize{  wangfy@tju.edu.cn}}
\begin{document}
\allowdisplaybreaks
\def\R{\mathbb R}  \def\ff{\frac} \def\ss{\sqrt} \def\B{\mathbf
B} \def\W{\mathbb W}
\def\N{\mathbb N} \def\kk{\kappa} \def\m{{\bf m}}
\def\ee{\varepsilon}\def\ddd{D^*}
\def\dd{\delta} \def\DD{\Delta} \def\vv{\varepsilon} \def\rr{\rho}
\def\<{\langle} \def\>{\rangle} \def\GG{\Gamma} \def\gg{\gamma}
  \def\nn{\nabla} \def\pp{\partial} \def\E{\mathbb E}
\def\d{\text{\rm{d}}} \def\bb{\beta} \def\aa{\alpha} \def\D{\scr D}
  \def\si{\sigma} \def\ess{\text{\rm{ess}}}
\def\beg{\begin} \def\beq{\begin{equation}}  \def\F{\scr F}
\def\Ric{\text{\rm{Ric}}} \def\Hess{\text{\rm{Hess}}}
\def\e{\text{\rm{e}}} \def\ua{\underline a} \def\OO{\Omega}  \def\oo{\omega}
 \def\tt{\tilde} \def\Ric{\text{\rm{Ric}}}
\def\cut{\text{\rm{cut}}} \def\P{\mathbb P} \def\ifn{I_n(f^{\bigotimes n})}
\def\C{\scr C}      \def\aaa{\mathbf{r}}     \def\r{r}
\def\gap{\text{\rm{gap}}} \def\prr{\pi_{{\bf m},\varrho}}  \def\r{\mathbf r}
\def\Z{\mathbb Z} \def\vrr{\varrho} \def\ll{\lambda}
\def\L{\scr L}\def\Tt{\tt} \def\TT{\tt}\def\II{\mathbb I}
\def\i{{\rm in}}\def\Sect{{\rm Sect}}  \def\H{\mathbb H}
\def\M{\scr M}\def\Q{\mathbb Q} \def\texto{\text{o}} \def\LL{\Lambda}
\def\Rank{{\rm Rank}} \def\B{\scr B} \def\i{{\rm i}} \def\HR{\hat{\R}^d}
\def\to{\rightarrow}\def\l{\ell}\def\iint{\int}
\def\EE{\scr E}\def\Cut{{\rm Cut}}
\def\A{\scr A} \def\Lip{{\rm Lip}}
\def\BB{\scr B}\def\Ent{{\rm Ent}}\def\L{\scr L}
\def\R{\mathbb R}  \def\ff{\frac} \def\ss{\sqrt} \def\B{\mathbf
B}
\def\N{\mathbb N} \def\kk{\kappa} \def\m{{\bf m}}
\def\dd{\delta} \def\DD{\Delta} \def\vv{\varepsilon} \def\rr{\rho}
\def\<{\langle} \def\>{\rangle} \def\GG{\Gamma} \def\gg{\gamma}
  \def\nn{\nabla} \def\pp{\partial} \def\E{\mathbb E}
\def\d{\text{\rm{d}}} \def\bb{\beta} \def\aa{\alpha} \def\D{\scr D}
  \def\si{\sigma} \def\ess{\text{\rm{ess}}}
\def\beg{\begin} \def\beq{\begin{equation}}  \def\F{\scr F}
\def\Ric{\text{\rm{Ric}}} \def\Hess{\text{\rm{Hess}}}
\def\e{\text{\rm{e}}} \def\ua{\underline a} \def\OO{\Omega}  \def\oo{\omega}
 \def\tt{\tilde} \def\Ric{\text{\rm{Ric}}}
\def\cut{\text{\rm{cut}}} \def\P{\mathbb P} \def\ifn{I_n(f^{\bigotimes n})}
\def\C{\scr C}      \def\aaa{\mathbf{r}}     \def\r{r}
\def\gap{\text{\rm{gap}}} \def\prr{\pi_{{\bf m},\varrho}}  \def\r{\mathbf r}
\def\Z{\mathbb Z} \def\vrr{\varrho} \def\ll{\lambda}
\def\L{\scr L}\def\Tt{\tt} \def\TT{\tt}\def\II{\mathbb I}
\def\i{{\rm in}}\def\Sect{{\rm Sect}}  \def\H{\mathbb H}
\def\M{\scr M}\def\Q{\mathbb Q} \def\texto{\text{o}} \def\LL{\Lambda}
\def\Rank{{\rm Rank}} \def\B{\scr B} \def\i{{\rm i}} \def\HR{\hat{\R}^d}
\def\to{\rightarrow}\def\l{\ell}
\def\8{\infty}\def\I{1}\def\U{\scr U}
\maketitle

\begin{abstract} Under  integrability conditions on distribution dependent coefficients, existence and uniqueness are proved for McKean-Vlasov type SDEs with non-degenerate noise. When the coefficients are Dini continuous in the space variable, gradient estimates and Harnack type inequalities are derived. These generalize the corresponding results derived for classical SDEs, and are new   in the distribution dependent setting.

\end{abstract} \noindent
 \vskip 2cm

\section{Introduction}
In order to  characterize  nonlinear Fokker-Planck equations using SDEs,   distribution dependent SDEs have been intensively investigated,
see \cite{SZ, MV} and references within for McKean-Vlasov type SDEs, and \cite{DV1,DV2, CA} and references within for Landau type equations. To ensure the existence and uniqueness of these type SDEs,  growth/regularity conditions  are used. On the other hand, however, due to Krylov's estimate and  Zvonkin's transform, the well-posedness of classical SDEs is proved under an integrability condition, which allows  the drift unbounded on compact sets.
The purpose of this paper is to extend this result to the distribution dependent situation, and to establish gradient estimates and Harnack type inequalities for the distributions under Dini continuity of the drift, which is much weaker than the Lipschitz  condition  used in \cite{FYW1, HRW}.

Let $\scr P$ be the set of all probability measures on $\R^d$. Consider the following distribution-dependent SDE on $\R^d$:
\beq\label{E1} \d X_t= b_t(X_t, \L_{X_t})\d t +\si_t(X_t, \L_{X_t})\d W_t,\end{equation}
where $W_t$ is the $d$-dimensional Brownian motion on a complete filtration probability space $(\OO,\{\F_t\}_{t\ge 0},\P)$, $\L_{X_t}$ is the law of $X_t$, and
$$b: \R_+\times\R^d\times \scr P\to \R^d,\ \ \si: \R_+\times\R^d\times \scr P\to \R^d\otimes\R^d$$ are measurable. When a different probability measure $\tt\P$ is concerned, we use $\L_\xi|\tt \P$ to denote the law of a random variable $\xi$ under the probability $\tt\P$.

By using a priori Krylov's estimate, a weak solution can be constructed  for \eqref{E1} by using  an approximation argument as in the classical setting, see \cite{GM} and references within. To prove the existence of strong solution,  we use  a fixed distribution $\mu_t$ to replace  the law of solution $\L_{X_t}$, so that the distribution SDE \eqref{E1}   reduces to the classical one. We   prove that when  the reduced SDE has strong uniqueness, the weak solution of  \eqref{E1} also provides a strong solution. We will then use   Zvonkin's transform to investigate the uniqueness, for which we first identify the distributions of given two solutions, so that these  solutions  solve the common reduced SDE, and thus, the pathwise uniqueness follows from existing argument developed for the classical SDEs.    However, there is essential difficulty to identify
the distributions of two solutions of \eqref{E1}. Once we have constructed the desired Zvonkin's transform for \eqref{E1} with singular coefficients, gradient estimates and Harnack type inequalities can be proved as in the regular situation considered in \cite{FYW1}.

The remainder of the paper is organized as follows. In Section 2 we summarize the main results of the paper. To prove these results,   some preparations are addressed in Section 3, including a new Krylov's estimate, two lemmas on weak convergence of stochastic processes, and a result on the existence of strong solutions for distribution dependent SDEs. Finally,  the main results are proved in Sections 4 and 5.


\section{Main results}

 We first recall Krylov's estimate in the study of SDEs.  We will fix a constant $T>0$, and only consider solutions of \eqref{E1} up to time $T$..
 For a measurable function $f$ defined on $[0,T]\times\mathbb{R}^d$, let
$$\|f\|_{L^q_p(s,t)}=\left(\int_s^t\left(\int_{\mathbb{R}^d}|f_r(x)|^p\d x\right)^{\frac{q}{p}}\d r\right)^{\frac{1}{q}}, \ \ p,q\ge 1, 0\le s\le t\le T. $$ When $s=0$, we simply denote   $\|f\|_{L^q_p(0,t)}=\|f\|_{L^q_p(t)}$. A key step in the study of singular SDEs is to establish Krylov type estimate (see for instance \cite{KR}).
For later use we introduce the following notion of $K$-estimate. 
We consider the following class of number pairs $(p,q)$:
   $$\scr K:=\Big\{(p,q)\in (1,\infty)\times(1,\infty):\   \ff d p +\ff 2 q<2\Big\}.$$

\beg{defn}[Krylov's Estimate]  \emph{An $\F_t$-adapted process $\{X_{s}\}_{0\le s\le T}$ is said to satisfy $K$-estimate, if   for any $(p,q)\in \scr K$, there exist    constants $\dd\in (0,1)$ and
$C>0$ such that for any nonnegative measurable function $f$ on $[0,T]\times \R^d$,
\beq\label{KR1} \E\bigg(\int_s^t f_r(X_r) \d r\Big| \F_s\bigg) \le C (t-s)^\dd \|f\|_{L_p^q(T)},\ \ \ 0\le s\le t\le T. \end{equation}}
\end{defn}
We note that \eqref{KR1} implies the following Khasminskii type estimate, see for instance \cite[Lemma 3.5]{XZ} and it's proof: there exists a constant $c>0$ such that
\beq\label{APP3} \E\bigg(\bigg(\int_s^t f_r(X_r) \d r\bigg)^n\Big| \F_s\bigg) \le c n! (t-s)^{\dd n}\|f\|_{L_p^q(T)}^n,\ \ \ 0\le s\le t\le T, \end{equation} and
for any $\ll>0$ there exists a constant $\LL=\LL(\ll,\dd,c)>0$  such that
\beq\label{KR2}   \E\big(\e^{\ll\int_0^T f_r(X_r) \d r}\big| \F_s\big) \le \e^{\LL  \left(1+\|f\|_{L_p^q(T)}\right)},\ \ s\in[0,T]. \end{equation}

Let $\theta\in [1,\infty)$, we will consider the SDE \eqref{E1} with initial distributions in the class
$$\scr P_\theta := \big\{\mu\in \scr P: \mu(|\cdot|^\theta)<\infty\big\}.$$    It is well known that
$\scr P_\theta$ is a Polish space under the Warsserstein distance
$$\W_\theta(\mu,\nu):= \inf_{\pi\in \C(\mu,\nu)} \bigg(\int_{\R^d\times\R^d} |x-y|^\theta \pi(\d x,\d y)\bigg)^{\ff 1 {\theta}},\ \ \mu,\nu\in \scr P_{\theta},$$ where $\C(\mu,\nu)$ is the set of all couplings of $\mu$ and $\nu$.  Moreover,   the topology induced by $\W_\theta$ on $\scr P_\theta$ coincides with the weak topology.

In the following three subsections,  we state our main results   on the existence, uniqueness and Harnack type inequalities respectively  for the distribution dependent SDE \eqref{E1}.

\subsection{Existence and uniqueness}

Let $$\scr P_\theta^a=\big\{\mu\in \scr P_\theta: \mu \text{\ is\ absolutely\ continuous\ with\ respect\ to the Lebesgue measure\ }\big\}.$$

To construct a weak solution of \eqref{E1} by using approximation argument as in \cite{GM, MV}, we need the following assumptions for some $\theta\ge 1$.

\beg{enumerate} \item[$(H^\theta)$]  There exists a sequence $(b^n,\sigma^n)_{n\ge 1}$, where
 $$b^n: [0,T]\times\R^d\times \scr P_\theta \to \R^d,\ \  \sigma^n: [0,T]\times\R^d\times \scr P_\theta \to \R^d\otimes \R^d $$
 are measurable, such that the following conditions hold:
 \item[$\ (1)$]  For $ \mu\in \scr P_\theta^a$ and $\mu^n\to \mu$ in $\scr P_\theta$,
$$  \lim_{n\to\infty}  \big\{ |b_t^n(x,\mu^n)-b_t(x,\mu)|+ \|\sigma^n_t(x,\mu^n)- \sigma_t(x,\mu)\| \big\} =0,\ \ \text{a.e.}\ \  (t,x)\in [0,T]\times\R^d.  $$
\item[$\ (2)$]  There exist $K>1$,   $(p,q)\in \scr K$ and  nonnegative $G\in  L_p^q(T)$
such that     for any $n\ge 1$,
 $$ |b_t^n(x,\mu)|^2\le G(t,x)+K,\ \ K^{-1} I \le (\sigma^n_t(\sigma^n_t)^\ast)(x,\mu)\le K I$$
 for all $(t,x,\mu)\in [0,T]\times \R^d\times \scr P_\theta.$
\item[$\ (3)$]    For each $n\ge 1$, there exists a constant $K_n>0$ such that $\|b^n\|_{\infty}\le K_n$ and
 \begin{equation}\label{con}\beg{split}
&|b_t^n(x,\mu)-b_t^n(y,\nu)|+ \|\sigma^n_t(x,\mu)- \sigma_t^n(y,\nu)\|\\
&\le K_n\big\{|x-y|+ \W_\theta(\mu,\nu)\big\},\ \  (t,x,y)\in [0,T]\times\R^d\times\R^d,\ \mu,\nu\in \scr P_\theta.\end{split}
\end{equation}
\end{enumerate}
 The main result in this part is the following.

\begin{thm}\label{T1.1} Assume $(H^\theta)$ for some constant $\theta\ge 1$.   Let $X_0$ be an $\F_0$-measurable   random variable on $\R^d$ with $\mu_0:=\L_{X_0}\in \scr P_\theta$. Then the following assertions hold.
\beg{enumerate} \item[$(1)$] The SDE $\eqref{E1}$ has a weak solution with initial distribution $\mu_0$ satisfying $ \L_{X_\cdot} \in C([0,T];\scr P_\theta)$ and the $K$-estimate.
\item[$(2)$] If $\sigma$ is uniformly continuous in $x\in\mathbb{R}^d$ uniformly with respect to $(t,\mu)\in[0,T]\times\scr P_{\theta},$ and
 for any $\mu_\cdot\in C([0,T]; \scr P_\theta)$, $b^\mu_t(x):= b_t(x, \mu_t)$ and $\si^\mu_t(x):= \si_t(x,\mu_t)$ satisfy
$| b^\mu|^2+\|\nn \si^\mu\|^2 \in L_p^q(T)$ for some $(p,q)\in \scr K$, where $\nn$ is the weak gradient in the space variable $x\in \R^d$, then the SDE \eqref{E1} has a strong solution satisfying $\L_{X_\cdot}\in C([0,T];\scr P_\theta)$ and the $K$-estimate.
\item[$(3)$] If, in addition to the condition in $(2)$,  there exists a  constant  $L\,>0$ such that
\beq\label{LIP} \|\si_t(x,\mu)-\si_t(x,\nu)\|+ |b_t(x,\mu)-b_t(x,\nu)|\le L\, \W_\theta(\mu,\nu)\end{equation} holds for all $ \mu,\nu\in \scr P_\theta$ and $ (t,x)\in [0,T]\times \R^d,$  then  the strong solution is unique. \end{enumerate}
\end{thm}

When $b$ and $\si$ do not depend on the distribution, Theorem \ref{T1.1} reduces back to the corresponding results derived for classical SDEs with singular coefficients, see for instance \cite{Z2}  and references within.

To  compare Theorem \ref{T1.1} with recent results on the existence and uniqueness of McKean-Vlasov type SDEs derived in \cite{CR, MV}, we consider a specific class of coefficients where the dependence on distributions is of integral type.  For $\mu\in \scr P$ and a (possibly multidimensional valued) real function $f\in L^1(\mu)$, let $\mu(f)= \int_{\R^d} f \d\mu$.  Let
\beq\label{EX1}  b_t(x,\mu):= B_t(x,\mu(\psi_b(t,x,\cdot)),\ \
  \si_t(x,\mu):= \Sigma_t(x,\mu(\psi_\si(t,x,\cdot))\end{equation} for $(t,x,\mu)\in [0,T]\times\R^d\times \scr P_\theta,$
where  for some $k \in \mathbb N,$
$$\psi_b, \psi_\si: [0,T]\times\R^d\times\R^d \to \R^k$$ are measurable and bounded
such that for some constant $\dd>0$,
\beq\label{EX2} |\psi_b(t,x,y)-\psi_b(t,x,y')|+ |\psi_\si(t,x,y)-\psi_\si(t,x,y')| \le \dd |y-y'| \end{equation} holds for all $ (t,x)\in [0,T]\times\R^d$ and $y,y'\in \R^d,$
and $$B: [0,T]\times\R^d\times\R^k \to \R^d,\ \ \Sigma: [0,T]\times\R^d\times\R^k \to \R^d\otimes\R^d$$  are measurable and continuous in the third variable in $\R^k$. We make the following assumption.
 \begin{enumerate}
\item[\bf (A)] Let $(b,\si)$  in $\eqref{EX1}$ for $(B,\Sigma)$ such that \eqref{EX2} holds, $B_t(x,\cdot)$ and $\Sigma_t(x,\cdot)$ are continuous for any $(t,x)\in [0,T]\times\R^d$. Moreover, there exist constant $K>1$,    $(p,q)\in \scr K$ and nonnegative $F\in L_p^q(T)$ such that
\beq\label{EX3} |b_t(x,\mu)|^2\le F(t,x)+K,\ \ K^{-1} I\le \si_t(x,\mu)\si_t(x,\mu)^* \le K I\end{equation} for all $(t,x,\mu)\in [0,T]\times\R^d\times \scr P_\theta.$
\end{enumerate}
\beg{cor}\label{C1.2}   Assume {\bf (A)}. Then the following assertions hold.
  \begin{enumerate}
\item[(1)] Assertion $(1)$ in Theorem $\ref{T1.1}$ holds.
\item[(2)] If moreover, $\sigma$ is uniformly continuous in $x\in\mathbb{R}^d$ uniformly with respect to $(t,\mu)\in[0,T]\times\scr P_{\theta},$ and
 for any $\mu_\cdot\in C([0,T]; \scr P_\theta)$, $b^\mu_t(x):= b_t(x, \mu_t)$ and $\si^\mu_t(x):= \si_t(x,\mu_t)$ satisfy
$| b^\mu|^2+\|\nn \si^\mu\|^2 \in L_p^q(T)$ for some $(p,q)\in \scr K$, where $\nn$ is the weak gradient in the space variable $x\in \R^d$, then assertion $(2)$ in Theorem $\ref{T1.1}$ hold.
\item[(3)] \ Besides the conditions in (2), if there exists a constant $c>0$ such that
$$|B_t(x,y)-B_t(x,y')| +\|\Sigma_t(x,y)-\Sigma_t(x,y')\|\le c |y-y'|,\ \ (t,x)\in [0,T]\times\R^d, y,y'\in  \R^k,$$
then for any     $\F_0$-measurable  random variable $X_0$ on $\R^d$ with $\mu_0:=\L_{X_0}\in \scr P_\theta$ for some $\theta\ge 1$, the SDE $\eqref{E1}$ has a unique strong solution with $\L_{X_\cdot}$ continuous in $\scr P_\theta.$
\end{enumerate}
\end{cor}
In the  next corollary   on the existence of weak solution we  do not  assume  \eqref{EX1}. This result will be used in Section 5.
\beg{cor}\label{C1.3}Assume that \eqref{LIP}, \eqref{EX3} hold. Then the SDE $\eqref{E1}$ has a weak solution with initial distribution $\mu_0$ satisfying $ \L_{X_\cdot} \in C([0,T];\scr P_\theta)$ and the $K$-estimate.
\end{cor}

We now explain that results in  Corollary \ref{C1.2} and Corollary \ref{C1.3} are new comparing with existing results on McKean-Vlasov SDEs. We first consider the model in \cite{CR} where $\psi_b$ and $\psi_\si$ are $\R$-valued functions such that
$$\|B\|_\infty +\sup_{(t,x,r)\in[0,T]\times\R^d\times\R}| \pp_r B_t(x, r)|<\infty,$$ $\psi_b$ is  H\"older continuous, $\psi_\si$ is Lipschitz continuous,  and for some constants $C>1$, $\theta\in (0,1]$,
\beg{align*} & C^{-1} I\le \Sigma \Sigma^*\le CI, \\
& \|\Sigma_t(x,r)- \Sigma_t(x',r')\| \le C(|x-x'| +|r-r'|), \\
&\|\pp_r \Sigma_t(x,r)- \pp_r \Sigma_t(x',r)\| \le C|x-x'|^\theta.\end{align*}
Then \cite[Theorem 1]{CR} says that when $\L_{X_0}\in \scr P_2$ the SDE \eqref{E1} has a unique strong solution. Obviously,   the above conditions imply
$\|b\|_\infty+\|\nn\si\|_\infty<\infty$, but this is not necessary for conditions in  Corollary \ref{C1.2} and Corollary \ref{C1.3}.

Next, \cite{MV} considers \eqref{E1} with
$$b_t(x,\mu):= \int_{\R^d} \tt b_t(x,y)\mu(\d y),\ \ \si_t(x,\mu):= \int_{\R^d} \tt\si_t(x,y) \mu(\d y)$$ for measurable functions  $$\tt b: [0,T]\times \R^d\times \R^d \to \R^d,\ \
\tt \si: [0,T]\times \R^d\times \R^d \to \R^d\otimes\R^d$$ satisfying
$$\|\tt \si_t(x,y)\|+|\tt b_t(x,y)|\le C(1+|x|),\ \ \tt\si\tt\si^*\ge C^{-1}I$$ for some constant $C>1.$ Then \cite[Theorem 1]{MV} says that when $\L_{X_0}\in \scr P_4$, \eqref{E1} has a weak solution. If moreover $\si$ does not depend on the distribution and $\|\nn \si\|_\infty<\infty$, then \cite[Theorem 2]{MV} shows that when $\E \e^{r|X_0|^2}<\infty$ for some $r>0$, the SDE \eqref{E1} has a unique strong solution. Obviously, to apply these results it is necessary that $b$ and $\nn \si$ are (locally) bounded, which is however not necessary for the condition in Corollary \ref{C1.2} and Corollary \ref{C1.3}.

\subsection{Harnack inequality}

In this subsection, we  investigate the dimension-free log-Harnack inequality introduced in \cite{RW10} for \eqref{E1}, see \cite{Wbook} and references within for general results on these type Harnack inequalities and applications.  We establish Harnack inequalities for $P_tf$ using  coupling by change of measures (see for instance \cite[\S 1.1]{Wbook}). To this end, we need to assume that the noise part is distribution-free; that is, we consider the following special version of \eqref{E1}:
\beq\label{E11} \d X_t= b_t(X_t,\L_{X_t})\d t +\si_t(X_t)\d W_t,\ \ t\in [0,T]\end{equation} As in \cite{FYW1}, we define $P_tf(\mu_0)$ and $P_t^*\mu_0$ as follows:
$$(P_tf)(\mu_0)= \int_{\R^d} f  \d(P_t^*\mu_0)= \E f(X_t(\mu_0)),\ \ f\in \B_b(\R^d), t\in [0,T],  \mu_0\in \scr P_2,$$ where $X_t(\mu_0)$ solves \eqref{E11} with
  $\L_{X_0}=\mu_0.$
Let
$$\D=\bigg\{\phi:  [0,\infty)\to [0,\infty)  \text{\ is\ increasing},  \phi^2 \text{\ is\ concave,} \int_0^1\ff{\phi(s)}s\d s<\infty\bigg\}.$$
We will need the following assumption.
\beg{enumerate}
\item[$\bf{(H)}$] $\|b\|_{\infty}<\infty$ and there exist a constant $K>1$ and $\phi\in \D$ such that for any
$t\in[0,T],\ x,y\in \mathbb{R}^d,$ and $\mu,\nu\in \scr P_2$,
\beq\label{H1}
K^{-1}I\leq(\sigma_t \sigma_t^{\ast})(x) \leq KI,\
\|\sigma_t(x)-\sigma_t(y)\|^2_{\mathrm{HS}}\leq K|x-y|^2,
\end{equation}
\beq\label{b-phi}
|b_t(x,\mu)-b_t(y,\nu)|\leq \phi(|x-y|)+K\mathbb{W}_2(\mu,\nu).
\end{equation}
\end{enumerate}
\beg{thm}\label{T3.1} Assume {\bf (H)}. There exists a  constant $C>0$ such that
  \beq\label{LH2}(P_{t}\log f)(\nu_0)\le  \log (P_{t}f)(\mu_0)+ \ff{C}{t\land 1}\W_2(\mu_0,\nu_0)^2\end{equation}
  for any $t\in (0,T],\mu_0,\nu_0\in\scr P_2,  f\in \B_b^+(\R^d)$ with $f\geq 1.$ Moreover, there exists a constant $p_0>1$ such that for any $p>p_0$,
   \beq\label{H2'}(P_{t}f)^p(\nu_0)\le  (P_{t}f^p)(\mu_0)\exp\left\{\ff{c}{t\land 1}\W_2(\mu_0,\nu_0)^2\right\}\end{equation}
   for any $t\in (0,T],\mu_0,\nu_0\in\scr P_2,  f\in \B_b^+(\R^d)$ and  some constant $c=c(p,K)>0$.
\end{thm}

\subsection{Shift Harnack inequality}

In this section we establish the shift Harnack inequality for $P_t$  introduced in \cite{W14a}. To this end, we assume that $\sigma_t(x,\mu)$ does not depend on $x$. So SDE \eqref{E1} becomes
\beq\label{E5} \d X_t= b_t(X_t, \L_{X_t})\d t +\si_t(\L_{X_t}) \d W_t,\ \ t\in [0,T].\end{equation}

\beg{thm}\label{T5.1} Let $\si: [0,T]\times \scr P_2\to \R^d\otimes \R^d$ and $b: [0,\infty)\times \R^d\times\scr P_2\to \R^d$ be  measurable such that
$\si$ is invertible with $\|\si_t\|_{\infty}+\|\si_t^{-1}\|_{\infty}$  is bounded in $t\in [0,T]$, and $b$ satisfies the corresponding conditions in {\bf (H)}.
\beg{enumerate} \item[$(1)$] For any $p>1, t\in [0,T], \mu_0\in \scr P_2, v\in\R^d$ and $f\in \B_b^+(\R^d)$,
\beg{align*}(P_{t}f)^p(\mu_0)\le &(P_{t}f^p(v+\cdot))(\mu_0)\\
&\times \exp\bigg[\ff{p\, \int_0^t \|\si_s^{-1}\|_{\infty}^2 \big\{|v|/t+\phi(s|v|/t)\big\}^2\d s}{2(p-1)}\bigg].\end{align*}
Moreover, for any $f\in \B_b^+(\R^d)$ with $f\geq 1$,
$$(P_{t}\log f)(\mu_0)\le \log (P_{t} f(v+\cdot))(\mu_0)+\frac{1}{2}\int_0^t \|\si_s^{-1}\|_{\infty}^2 \big\{|v|/t+\phi(s|v|/t)\big\}^2\d s.$$
\end{enumerate} \end{thm}

\section{Preparations}

We first present a new result on Krylov's estimate, then recall two lemmas from \cite{GM} for the construction of weak solution, and finally introduce two lemmas on the existence and uniqueness of strong solutions.

\subsection{Krylov's estimate}

Consider the following SDE on $\R^d$:
\beq\label{EN} \d X_t= b_t(X_t)\d t + \si_t(X_t)\d W_t,\ \ t\in [0,T].\end{equation}
\beg{lem}\label{KK} Let $T>0$, and let $p,q\in (1,\infty)$ with $\ff d p +\ff 2 q<1$.  Assume that $\sigma_t(x)$ is uniformly continuous in $x\in\mathbb{R}^d$ uniformly with respect to $t\in[0,T]$, and that for a constant $K>1$ and some nonnegative function $F\in L_p^q(T)$
such that \beq\label{APP1} K^{-1}I\le \si_t(x)\si_t(x)^*\le K I,\ \ (t,x)\in [0,T]\times \R^d,\end{equation}
 \beq\label{APP2} |b_t(x)| \le K+ F(t,x),\ \ (t,x)\in [0,T]\times \R^d.\end{equation}
 Then for any $(\aa,\bb)\in \scr K$, there exist constants  $C=C(\dd,K, \aa,\bb, \|F\|_{L_p^q(T)})>0$  and $\dd=\dd(\aa,\bb)>0$, such that for any $s\in [0,T)$ and any solution $(X_{s,t})_{t\in [s,T]}$ of $\eqref{EN}$ from time $s$,
\beq\label{APP'}\E\bigg[\int_s^t |f|(r, X_{s,r}) \d r\Big| \F_s\bigg]\le   C  (t-s)^{\dd}\|f\|_{L_{\aa}^{\bb}(T)},\ t\in [s,T], f\in L_{\aa}^{\bb}(T).\end{equation}
 \end{lem}
 \beg{proof}
  When $b$ is bounded, the assertion is due to \cite[Theorem 2.1]{Z2}. If $|b|\leq K+F$ for some constant $K>0$ and $0\leq F\in L_p^q(T)$,  then we have a decomposition $b=b^{(1)}+b^{(2)}$ with $\|b^{(1)}\|_{\infty}\leq K$ and $|b^{(2)}|\leq F$, for instance, $b^{(1)}=\frac{b}{1\vee(|b|/K)}$. Letting the diffeomorphisms $\{\theta_t\}_{t\in[0,T]}$ on $\mathbb{R}^d$ be constructed in \cite[Lemma 4.3]{Z2} for $b^{(2)}$ replacing $b$, then $Y_{s,t}=\theta_t(X_{s,t})$ solves
  \beq\label{EN'} \d Y_t= \bar{b}_t(Y_t)\d t + \bar{\si}_t(Y_t)\d W_t,\ \ t\in [s,T],\end{equation}
  where $\bar{b}$ is bounded, and $\bar{\si}$ is uniformly continuous in $x\in\mathbb{R}^d$ uniformly with respect to $t\in[0,T]$. Moreover, there exists a constant $\bar{K}>1$ depending on $K$ and $\|F\|_{L_p^q(T)}$ such that \beq\label{APP1'} \bar{K}^{-1}I\le \bar{\si}_t(x)\bar{\si}_t(x)^*\le \bar{K} I,\ \ (t,x)\in [0,T]\times \R^d,\end{equation} and $$\|\bar{b}\|_{\infty}+\|\nabla \theta\|_{\infty}+\|\nabla \theta^{-1}\|_{\infty}\leq \bar{ K}.$$ Again by \cite[Theorem 2.1]{Z2}, there exists a constant $C=C(\dd,\bar{K}, \aa,\bb)>0$ and $\dd=\dd(\aa,\bb)>0$ such that
\beq\label{APP3'}\E\bigg[\int_s^t |f|(r, Y_{s,r}) \d r\Big| \F_s\bigg]\le   C  (t-s)^{\dd}\|f\|_{L_{\aa}^{\bb}(T)},\ t\in [s,T], f\in L_{\aa}^{\bb}(T).\end{equation}
This together with $\|\nabla \theta\|_{\infty}<\bar{K}$  implies that
\begin{align*}&\E\bigg[\int_s^t |f|(r, X_{s,r}) \d r\Big| \F_s\bigg]
 =\E\bigg[\int_s^t |f|(r, \theta_r^{-1}(Y_{s,r})) \d r\Big| \F_s\bigg]\\
&\leq C  (t-s)^{\dd}\left(\int_0^T\left(\int_{\mathbb{R}^d}|f(r,\theta^{-1}_r(x))|^\alpha\d x\right)^{\frac{\beta}{\alpha}}\d r\right)^{\frac{1}{\beta}}\\
&=C  (t-s)^{\dd}\left(\int_0^T\left(\int_{\mathbb{R}^d}|f(r,y)|^\alpha|\mathrm{det} \nabla\theta_r|\d y\right)^{\frac{\beta}{\alpha}}\d r\right)^{\frac{1}{\beta}}\\
&\le   C  (t-s)^{\dd}\|f\|_{L_{\aa}^{\bb}(T)},\ t\in [s,T], f\in L_{\aa}^{\bb}(T).\end{align*}
Then the proof is finished.
 \end{proof}

\subsection{Convergence  of stochastic processes}
To prove Theorem \ref{T1.1}(1), we will use the following two lemmas due to \cite[Lemma 5.1, 5.2]{GM}.
\begin{lem}\label{PC} Let $\{\psi^n\}_{n\geq 1}$ be a sequence of $d$-dimensional processes defined on
some probability space. Assume that
\begin{align}\label{Ub}\lim_{R\to\infty}\sup_{n\geq 1}\sup_{t\in[0,T]}\P(|\psi^n_t|>R)=0,
\end{align}
and for any $\vv>0$,
\begin{align}\label{ETC}
\lim_{\theta\to0}\sup_{n\geq 1}\sup_{s,t\in[0,T]}\{\P(|\psi^n_t-\psi^n_s|>\varepsilon): |t-s|\leq \theta\}=0.
\end{align}
  Then there exist a sequence $\{n_k\}_{k\geq 1}$, a probability space $(\tilde{\Omega},\tilde{\F}, \tilde{\P})$ and
stochastic processes $\{X_t, X^k_t\}_{t\in [0,T]} (k \geq  1)$, such that for every $t\in [0,T]$, $\L_{\psi^{n_k} _t}|\P=\L_{X^k_t}|\tilde{\P}$, and $X^k_t$ converges to $X_t$ in probability
$\tt\P$ as $k\to\infty$.
\end{lem}

\begin{lem}\label{SL} Let $\{\eta^n\}_{n\geq 1}$ and $\eta$ be uniformly bounded $\mathbb{R}^{d}\otimes\mathbb{R}^{k}$-valued stochastic
processes, and let $W^n_t$ and $W_t$  for $t\in [0,T]$ be Wiener processes such that the stochastic It\^{o} integrals
$$I^n_t:=\int^t
_0 \eta^n_s \d W^n_s,\ \ I_t :=
\int ^t
_0 \eta_s  \d W_s,\ \ t\in [0,T]$$ are well-defined. Assume
that $\eta^n_t \to \eta_t$ and $W^n_t \to W_t$ in probability for every $t\in [0,T]$. Then
$$\lim_{n\to\infty}\P\left(\sup_{t\in[0,T]}|I^n_t-I_t|\geq\varepsilon\right)=0,\ \ \vv>0.$$
 \end{lem}

\subsection{Existence and uniqueness on strong solutions}
We first present a result on the existence of strong solutions deduced from weak solutions, then  introduce a result on the existence and uniqueness of strong solutions under a Lipschitz type condition.

\begin{lem}\label{SS} Let $(\bar\Omega, \bar\F_t,\bar W_t,\bar\P)$ and $\bar{X}_t$ be a weak solution to \eqref{E1} with $\mu_t:=\L_{\bar X_t}|\bar\P=\mu_t$. If the SDE
\begin{align}\label{class}
\d X_t= b_t(X_t,\mu_t)\,\d t+ \sigma_t(X_t,\mu_t)\,\d W_t,\ \ 0\le t\le T
\end{align} has a unique strong solution $X_t$ up to life time with $\L_{X_0}=\mu_0$, then   \eqref{E1} has a strong solution.
\end{lem}
\begin{proof} Since $\mu_t= \scr L_{\bar X_t}|\bar \P$,   $\bar{X}_t$ is a weak solution to \eqref{class}. By Yamada-Watanabe principle, the strong uniqueness of \eqref{class} implies the weak uniqueness, so that $X_t$ is nonexplosive with    $\L_{X_t}=\mu_t, t\ge 0$. Therefore, $X_t$ is a strong solution to \eqref{E1}.
\end{proof}

\begin{lem}\label{SS2} Let $\theta\ge 1$ and $\delta_0$ be the Dirac measure at point $0$. If $b_t(0,\delta_0)$ is bounded in $t\in[0,T]$, and there exists a constant $L>0$ such that
\beq\label{LIPS} \beg{split} &\|\si_t(x,\mu)-\si_t(y,\nu)\|+|b_t(x,\mu)-b_t(y,\nu)|\\
&\le L\big\{|x-y|+\W_\theta(\mu,\nu)\big\},\ \ x,y\in\R^d, \mu,\nu\in \scr P_\theta, t\in [0,T],\end{split}\end{equation}
then for any $X_0$ with $\E |X_0|^\theta<\infty$,   \eqref{E1} has a unique strong solution $(X_t)_{t\in [0,T]}$.
\end{lem}
\beg{proof} When $\theta\ge 2$ the assertion follows from \cite[Theorem 2.1]{FYW1}. So we only consider $\theta<2$. As explained in \cite{FYW1} that it suffices to find a constant $t_0\in (0,T)$ independent of $X_0$ such that \eqref{E1} has a unique strong solution up to time $t_0$ and
$\sup_{t\in [0,t_0]} \E |X_t|^\theta<\infty$.

Let $X_t^{(0)}=X_0 $ and $\mu_t^{(0)}=\mu_0$ for $t\in [0,T].$ For any $n\ge 1$, consider the SDE
$$\d X_t^{(n)}= b_t(X_t^{(n)}, \mu_t^{(n-1)})\d t+ \si_t(X_t^{(n)}, \mu_t^{(n-1)})\d W_t,\ \ X_0^{(n)}=X_0,$$
where $\mu_t^{(n-1)}=\L_{X_t^{(n-1)}}, 0\le t\le T.$ By \cite[Lemma 2.3(1)]{FYW1}, for any $n\ge 1$ this SDE has a unique solution and
\beq\label{SS3} \sup_{s\in [0,T]} \E |X_s^{(n)}|^\theta<\infty,\ \ n\ge 1.\end{equation}
Moreover,   letting
$$\xi_t^{(n)}:= X_t^{(n+1)}- X_t^{(n)},\ \ \LL_t^{(n)}:= \si_t(X_t^{(n+1)},\mu_t^{(n)})- \si_t(X_t^{(n)},\mu_t^{(n-1)}),$$
\cite[(2.11)]{FYW1} implies
$$\d |\xi_t^{(n)}|^2 \le 2 \<\LL_t^{(n)}\d W_t, \xi_t^{(n)}\> + K_0  \big\{|\xi_t^{(n)}|^2 + \W_\theta(\mu_t^{(n)}, \mu_t^{(n-1)})^2\big\}\d t,\ \ n\ge 1, t\in [0,T] $$
 for some constant $K_0>0$. Since $\xi_0^{(n)}=0$,  it follows that
\beg{align*} \E|\xi_t^{(n)}|^2 &\le \int_0^t K_0 \e^{K_0(t-s)} \W_\theta(\mu_s^{(n)},\mu_s^{(n-1)})^2\d s\\
&\le t K_0\e^{K_0T} \sup_{s\in [0,t]}\big(\E |\xi_t^{(n-1)}|^\theta\big)^{\ff 2 \theta},  \ \ t\in [0,T], n\ge 1.\end{align*} Since $\theta<2$, by Jensen's inequality we may find out a constant $K_1>0$ such that
$$\sup_{s\in [0,t]}\E |\xi_s^{(n)}|^\theta\le K_1 t^{\ff  \theta 2}  \sup_{s\in [0,t]}\E |\xi_s^{(n-1)}|^\theta,\ \ n\ge 1, t\in [0,T].$$
So, taking $t_0\in (0, T\land K_1^{-\ff 2 \theta})$, we may find a constant $\vv\in (0,1)$ such that
$$\sup_{s\in [0,t]}\E |\xi_s^{(n)}|^\theta\le \vv^{n} \sup_{s\in [0,t_0]} \E |X_s^{(1)}-X_0|^\theta<\infty,\ \ n\ge 1,\in[0,t_0].$$
Therefore, for any $t\in [0,t_0]$ there exists an $\F_t$-measurable random variable $X_t$ on $\R^d$ such that
$$\lim_{n\to\infty} \sup_{t\in [0,t_0]}\W_\theta(\mu_t^{(n)},\mu_t)^\theta\le \lim_{n\to\infty} \sup_{t\in [0,t_0]} \E |X_t^{(n)}- X_t|^\theta =0,$$ where
 $\mu_t:=\L_{X_t}$. Combining this with \eqref{LIPS} and letting $n\to\infty$ in the equation
$$X_t^{(n)}= \int_0^t b_s(X_s^{(n)}, \mu_s^{(n-1)}) \d s +\int_0^t \si_s(X_s^{(n)}, \mu_s^{(n-1)}) \d W_s,\ \ n\ge 1, t\in [0,t_0],$$
  we derive for every $t\in [0,t_0]$,
  $$X_t= \int_0^t b_s(X_s, \mu_s) \d s +\int_0^t \si_s(X_s, \mu_s) \d W_s.$$ Thus, $(X_s)_{s\in [0,t_0]}$ has a continuous version which is a strong solution of \eqref{E1} up to time $t_0$. The uniqueness is trivial by using condition \eqref{LIPS} and It\^o's formula. \end{proof}

\section{Proofs of Theorem \ref{T1.1} and Corollary \ref{C1.2}}
\subsection{Proof of Theorem \ref{T1.1}(1)-(2)}

  According to \cite{Z2},  the condition in Theorem \ref{T1.1}(2) implies that the SDE \eqref{class} has a unique strong solution. So,
  by Lemma \ref{SS}, Theorem \ref{T1.1}(2) follows from Theorem \ref{T1.1}(1).  Below we only prove the existence of weak solution.

 By Lemma \ref{SS2}, condition (3) in $(H^\theta)$ implies that the SDE
 \beq\label{X^n} \d X^n_t= b_t^n(X^n_t,\L_{X_t^n})\d t + \si_t^n(X^n_t, \L_{X_t^n}) \d W_t,\ \ X_0^n= X_0\end{equation}
 has a unique strong solution $(X^n_t)_{t\in [0,T]}$.   So,   Lemma \ref{KK},  \eqref{con} and condition (2) in  $(H^\theta)$ imply that for any  $(p,q)\in \scr K$,
\beq\label{KRE} \E  \int_s^t f(r,X_r^n)\d r \le C(t-s)^\dd \|f\|_{L_p^q(T)},\ \ 0\le f\in L_p^q(T), n\ge 1\end{equation} holds for some constants $C>0$ and $ \dd\in (0,1).$

 We first show that Lemma \ref{PC} applies to $\psi_n:=(X^n,W)$, for which it suffices to verify conditions \eqref{Ub} and \eqref{ETC} for $\psi_n:=X^n$. By condition (2) in $(H^\theta)$ and \eqref{APP3} implied by \eqref{APP'}, there exist  constants $c_1,c_2>0$ such that
\begin{equation}\label{AP3}\beg{split}
 \E |X  ^n_t|^\theta &\leq c_1\bigg\{\E |X_0|^\theta+\E\bigg(\int_0^T|b^n_t(X^n_t,\L_{X^n_t})|\,\d t\bigg)^\theta\\
 &\ \ \ \ \ \ \ \ \  +  \E\left(\int_0^T\|\sigma^n_t(X^n_t,\L_{X^n_t})\|^2\,\d t\right)^\frac{\theta}{2}\bigg\}\\
&\leq c_2\Big(\E |X_0|^\theta + T^\theta+ \|G\|_{L^{q}_{p}(T)}^\theta + T^{\ff \theta 2}\Big) <\infty,  \ \ n\ge 1, t\in [0,T].
\end{split}\end{equation} Thus,  \eqref{Ub} holds for $\psi_n:=X^n.$

Next, by the same reason, there exists a constant $c_3>0$ such that for any $0 \leq s \leq t \leq T$,
\begin{align*}
\E |X^n_t-X^n_s|&\leq \E\int_s^t|b^n_r(X^n_r,\L_{X^n_r})|\,\d r+  \E\left(\int_s^t\|\sigma^n_r(X^n_r,\L_{X^n_r})\|^2\,\d r\right)^\frac{1}{2}\\
&\leq c_3\big(t-s + (t-s)^\dd\|G\|_{L^q_p(T)}+ (t-s)^{\frac{1}{2}}\big).
\end{align*}Hence, \eqref{ETC}   holds for $\psi_n:=X^n$.
According to  Lemma \ref{PC}, there exists a subsequence of $(X^n,W)_{n\ge 1}$, denoted again by $(X^n,W)_{n\ge 1}$, stochastic processes $(\tilde{X}^n,\tilde{W}^n)_{n\ge 1}$ and  $(\tilde{X}, \tilde{W})$ on a complete probability space $(\tilde{\OO}, \tilde{\F}, \tilde{\P})$ such that $\L_{(X^n, W)}|\P=\L_{(\tilde{X}^n, \tilde{W}^n)}|\tilde{\P}$ for any $n\geq 1$, and for any $t\in [0,T]$, $\lim_{n\to\infty}(\tilde{X}^n_t, \tt W_t^n)=(\tilde{X}_t, \tt W_t)$ in the probability $\tt\P$. As in \cite{GM}, let $\tilde{\F}^n_{t}$
be  the completion of the $\si$-algebra generated by the  $\{\tilde{X}^n_s, \tilde{W}^n_s: s\leq t\}$. Then as shown in \cite{GM},   $\tilde{X}^n_t$ is $\tilde{\F}^n_{t}$-adapted and continuous (since $X^n$ is continuous and $\L_{X^n}|\P=\L_{\tt X^n}|\tt\P$),   $\tilde{W}^n$
is a $d$-dimensional Brownian motion on $(\tt \OO, \{\tt \F_t^n\}_{t\in [0,T]},\tt\P)$,    and $(\tilde{X}^n_t,\tilde{W}^n_t)_{t\in [0,T]}$ solves the SDE
\beq\label{titlde-X^n}
\d \tilde{X}^n_t= b^n_t(\tilde{X}^n_t,\L_{\tilde{X}^n_t}|\tilde{\P})\,\d t+ \sigma^n_t(\tilde{X}^n_t,\L_{\tilde{X}^n_t}|\tilde{\P})\,\d \tilde{W}^n_t,\ \ \L_{\tilde{X}^n_0}|\tilde{\P}=\L_{X_0}|\P.
\end{equation}
  Simply denote $\L_{\tilde{X}^n_t}|\tilde{\P}=\L_{\tilde{X}^n_t}$ and $\L_{\tilde{X}_t}|\tilde{\P}=\L_{\tilde{X}_t}$.  Then $(\tilde{X_t},\tilde{W_t})_{t\in [0,T]}$ is a weak solution to \eqref{E1} provided   for any   $\varepsilon>0$,
\beq\label{(A1)} \lim_{n\to\infty}\tilde{\P}\left(\sup_{s\in[0,T]}\int_{0}^s| b^n_t(\tilde{X}^n_t,\L_{\tilde{X}^n_t})-b_t(\tilde{X}_t,\L_{\tilde{X}_t})|\,\d t\geq\varepsilon\right)=0,\end{equation}
and
\beq\label{(A2)} \lim_{n\to\infty}\tilde{\P}\left(\sup_{s\in[0,T]}\left| \int_{0}^s\si^n_t(\tilde{X}^n_t,\L_{\tilde{X}^n_t})\d\tilde{W}^n_t-\int_{0}^s\si_t(\tilde{X}_t,\L_{\tilde{X}_t})\,\d \tilde{W}_t\right|\geq\varepsilon\right)=0.\end{equation}
In the following  we   prove these two limits respectively.

\beg{proof}[Proof of \eqref{(A1)}]  For any $n\ge m\ge 1$, we have
$$\int_{0}^s | b^n_t({\tilde X}^n_t,\L_{{\tilde X}^n_t})-b_t({\tilde X}_t,\L_{{\tilde X}_t})|\,\d t\le I_1(s)+ I_2(s)+I_3(s),$$ where
\begin{align*}
&I_1(s):= \int_{0}^s| b^n_t({\tilde X}^n_t,\L_{{\tilde X}^n_t})-b^{m}_t({\tilde X}^n_t,\L_{{\tilde X}_t})|\,\d t,\\
&I_2(s):=\int_{0}^s|b^{m}_t({\tilde X}^n_t,\L_{{\tilde X}_t})-b^{m}_t({\tilde X}_t,\L_{{\tilde X}_t})|\,\d t,\\
&I_3(s):= \int_{0}^s| b^{m}_t({\tilde X}_t,\L_{{\tilde X}_t})-b_t({\tilde X}_t,\L_{{\tilde X}_t})|\,\d t.
\end{align*}
Below we estimate these $I_i(s)$ respectively.

Firstly, by Chebyshev's inequality and \eqref{KRE}, we arrive at
\begin{align*}
\tilde{\P}(\sup_{s\in[0,T]}I_1(s)\geq\frac{\varepsilon}{3})&\leq \frac{9}{\varepsilon^2}\E\int_{0}^T1_{\{|\tilde{X}^n_t|\leq R\}}| b^n_t(\tilde{X}^n_t,\tilde{\mu}^n_t)-b^{m}_t(\tilde{X}^n_t,\tilde{\mu}_t)|^2\,\d t\\
&+\frac{9}{\varepsilon^2}\E\int_{0}^T1_{\{|\tilde{X}^n_t|> R\}}| b^n_t(\tilde{X}^n_t,\tilde{\mu}^n_t)-b^{m}_t(\tilde{X}^n_t,\tilde{\mu}_t)|^2\,\d t\\
&\leq \frac{9C}{\varepsilon^2}\left(\int_{0}^T\left(\int_{|x|\leq R}|b^n_t(x,\tilde{\mu}^n_t)-b^{m}_t(x,\tilde{\mu}_t)|^{2p}\d x\right)^{q/p}\d t\right)^{\frac{1}{q}}\\
&+\frac{36K}{\varepsilon^2}\int_{0}^T\tilde{\P}(|\tilde{X}^n_t|> R)\d t+\frac{36C}{\varepsilon^2}\|G1_{\{|\cdot|>R\}}\|_{L_p^q(T)}.
\end{align*}
Since $\tilde{X}^n_t$ converges to $\tilde{X}_t$ in probability,  \eqref{AP3} implies
$$\lim_{n\to\infty} \W_\theta(\tt\mu_t^n,\mu_t) =0,$$
and $$\lim_{n\to\infty}\tilde{\P}(|\tilde{X}^n_t|> R)\leq\tilde{\P}(|\tilde{X}_t|\geq R).$$
Then it follows from $(H^\theta)$ (1) and (3) that
\begin{align*}
&\lim_{n\to\infty}|b^n_t(x,\tilde{\mu}^n_t)-b_t(x,\tilde{\mu}_t)|=0, \ \ a.e. \ \ t\in[0,T],x\in\mathbb{R}^d.
\end{align*}
So, by condition (2) in $(H^\theta)$, we may apply the dominated convergence theorem to derive
\begin{equation}\begin{split}\label{I1J}
&\limsup_{n\to\infty}\tilde{\P}(\sup_{s\in[0,T]}I_1(s)\geq\frac{\varepsilon}{4})\\
&\leq \frac{9C}{\varepsilon^2}\left(\int_{0}^T\left(\int_{|x|\leq R}|b_t(x,\tilde{\mu}_t)-b^{m}_t(x,\tilde{\mu}_t)|^{2p}\d x\right)^{q/p}\d t\right)^{\frac{1}{q}}\\
&+\frac{36K}{\varepsilon^2}\int_{0}^T\tilde{\P}(|\tilde{X}_t|\geq R)\d t+\frac{36C}{\varepsilon^2}\|G1_{\{|\cdot|>R\}}\|_{L_p^q(T)}.
\end{split}\end{equation}
Since $b^{m}$ is bounded and continuous, it follows that
\begin{align*}
&\limsup_{n\to\infty}\tilde{\P}\Big(\sup_{s\in[0,T]}I_2(s)\geq\frac{\varepsilon}{3}\Big)\leq\limsup_{n\to\infty}\frac{3}{\varepsilon}
\E\int_{0}^T|b^{m}_t(\tilde{X}^n_t,\L_{\tilde{X}_t})-b^{m}_t(\tilde{X}_t,\L_{\tilde{X}_t})|\,\d t=0.
\end{align*}

Finally, since $\tt X^n_t\to \tt X_t$ in probability, estimate \eqref{KRE} also holds for $\tt X$ replacing $\tt X^n$.   Therefore, inequality \eqref{I1J} holds for $I_3$ replacing $I_1$. In conclusion, we arrive at
\begin{align*}
&\limsup_{n\to\infty}\tilde{\P}\Big(\sup_{s\in[0,T]}\int_{0}^s| b^n_t(\tilde{X}^n_t,\L_{\tilde{X}^n_t})-b_t(\tilde{X}_t,\L_{\tilde{X}_t})|\,\d t\geq\varepsilon\Big)\\
&\leq \limsup_{n\to\infty}\sum_{i=1}^3\tilde{\P}\Big(\sup_{s\in[0,T]}I_i(s)\geq\frac{\varepsilon}{3}\Big)\\
&\leq \frac{18C}{\varepsilon^2}\left(\int_{0}^T\left(\int_{|x|\leq R}|b_t(x,\tilde{\mu}_t)-b^{m}_t(x,\tilde{\mu}_t)|^{2p}\d x\right)^{q/p}\d t\right)^{\frac{1}{q}}\\
&+\frac{72K}{\varepsilon^2}\int_{0}^T\tilde{\P}(|\tilde{X}_t|\geq R)\d t+\frac{72C}{\varepsilon^2}\|G1_{\{|\cdot|>R\}}\|_{L_p^q(T)}.
\end{align*}
for any $m>0$ and $R>0$. Then letting first $m\to\infty$ and then $R\to\infty$,
due to (1) and (2) in $(H^\theta)$, we obtain from the dominated convergence theorem that
\begin{align*}
&\limsup_{n\to\infty}\tilde{\P}\Big(\sup_{s\in[0,T]}\int_{0}^s| b^n_t(\tilde{X}^n_t,\L_{\tilde{X}^n_t})-b_t(\tilde{X}_t,\L_{\tilde{X}_t})|\,\d t\geq\varepsilon\Big)=0.
\end{align*}\end{proof}

\beg{proof}[Proof of \eqref{(A2)}]  For any $n\ge m\ge 1$ we have
\begin{align*}
&\left| \int_{0}^s\si^n_t(\tilde{X}^n_t,\L_{\tilde{X}^n_t})\d\tilde{W}^n_t-\int_{0}^s\si_t(\tilde{X}_t,\L_{\tilde{X}_t})\,\d \tilde{W}_t\right|\\
&\leq\left| \int_{0}^s\si^n_t(\tilde{X}^n_t,\L_{\tilde{X}^n_t})\d\tilde{W}^n_t-\int_{0}^s\si^{m}_t(\tilde{X}^n_t,\L_{\tilde{X}^{m}_t})\,\d \tilde{W}^n_t\right|\\
&+\left| \int_{0}^s\si^{m}_t(\tilde{X}^n_t,\L_{\tilde{X}^{m}_t})\d\tilde{W}^n_t-\int_{0}^s\si^{m}_t(\tilde{X}_t,\L_{\tilde{X}^{m}_t})\,\d \tilde{W}_t\right|\\
&+\left| \int_{0}^s\si^{m}_t(\tilde{X}_t,\L_{\tilde{X}^{m}_t})\d\tilde{W}_t-\int_{0}^s\si_t(\tilde{X}_t,\L_{\tilde{X}_t})\,\d \tilde{W}_t\right|\\
&=:J_1(s)+J_2(s)+J_3(s).
\end{align*}
By Chebyshev's inequality, BDG inequality  and \eqref{KRE},    we have
\begin{align*}
\tilde{\P}\Big(\sup_{s\in[0,T]}J_1(s)\geq\frac{\varepsilon}{3}\Big)&\leq \frac{9}{\varepsilon^2}\E\int_{0}^T1_{\{|\tilde{X}^n_t|\leq R\}}\| \sigma^n_t(\tilde{X}^n_t,\L_{\tilde{X}^n_t})-\sigma^{m}_t(\tilde{X}^n_t,\L_{\tilde{X}^{m}_t})\|_{HS}^2\,\d t\\
&+ \frac{9}{\varepsilon^2}\E\int_{0}^T1_{\{|\tilde{X}^n_t|>R\}}\| \sigma^n_t(\tilde{X}^n_t,\L_{\tilde{X}^n_t})-\sigma^{m}_t(\tilde{X}^n_t,\L_{\tilde{X}^{m}_t})\|_{HS}^2\,\d t\\
&\leq \frac{9C}{\varepsilon^2}\left(\int_{0}^T\left(\int_{|x|\leq R}\|\si^n_t(x,\tilde{\mu}^n_t)-\si^{m}_t(x,\tilde{\mu}^{m}_t)\|_{HS}^{2p}\d x\right)^{\frac{q}{p}}\d t\right)^{\frac{1}{q}}\\
&+\frac{18dK}{\varepsilon^2}\int_{0}^T\tilde{\P}(|\tilde{X}^n_t|> R)\d t.
\end{align*}
By condition (1) in $(H^\theta)$, and $\tt\mu^n_t\to\tt\mu_t$ in $\scr P_\theta$ as observed above, we have
\begin{align*}
&\lim_{n\to\infty}\|\si^n_t(x,\tilde{\mu}^n_t)-\si_t(x,\tilde{\mu}_t)\|=0,
\end{align*}
and $$\lim_{n\to\infty}\tilde{\P}(|\tilde{X}^n_t|> R)\leq\tilde{\P}(|\tilde{X}_t|\geq R).$$
So,   the dominated convergence theorem gives
\begin{equation}\begin{split}\label{I_1}
&\limsup_{n\to\infty}\tilde{\P}\Big(\sup_{s\in[0,T]}J_1(s)\geq\frac{\varepsilon}{3}\Big)\\
&\leq \frac{9C}{\varepsilon^2}\left(\int_{0}^T\left(\int_{|x|\leq R}\|\si_t(x,\tilde{\mu}_t)-\si^{m}_t(x,\tilde{\mu}^{m}_t)\|_{HS}^{2p}\d x\right)^{\frac{q}{p}}\d t\right)^{\frac{1}{q}}\\
&+\frac{18dK}{\varepsilon^2}\int_{0}^T\tilde{\P}(|\tilde{X}_t|> R)\d t.
\end{split}\end{equation}
Similarly,
\begin{align*}
& \tilde{\P}\Big(\sup_{s\in[0,T]}J_3(s)\geq\frac{\varepsilon}{3}\Big)\\
&\leq \frac{9C}{\varepsilon^2}\left(\int_{0}^T\left(\int_{|x|\leq R}\|\si_t(x,\tilde{\mu}_t)-\si^{m}_t(x,\tilde{\mu}^{m}_t)\|_{HS}^{2p}\d x\right)^{\frac{q}{p}}\d t\right)^{\frac{1}{q}}\\
&+\frac{18dK}{\varepsilon^2}\int_{0}^T\tilde{\P}(|\tilde{X}_t|> R)\d t.
\end{align*}
So, applying Lemma \ref{SL} to
\begin{align*}
\eta_n(t):=\si^{m}_t(\tilde{X}^n_t,\tilde{\mu}^{m}_t), \ \ \eta(t):=\si^{m}_t(\tilde{X}_t,\tilde{\mu}^{m}_t),
\end{align*}
we conclude that when $n\to\infty$,
$$ \int_{0}^s\si^{m}_t(\tilde{X}^n_t,\L_{\tilde{X}^{m}_t})\d\tilde{W}^n_t\to\int_{0}^s\si^{m}_t(\tilde{X}_t,\L_{\tilde{X}^{m}_t})\,\d \tilde{W}_t$$
in probability $\tt \P$, uniformly in $s\in[0, T ]$. Hence,
\begin{align*}
&\lim_{n\to\infty}\tilde{\P}\left(\sup_{s\in[0,T]}\left| \int_{0}^s\si^n_t(\tilde{X}^n_t,\L_{\tilde{X}^n_t})\d\tilde{W}^n_t-\int_{0}^s\si_t(\tilde{X}_t,\L_{\tilde{X}_t})\,\d \tilde{W}_t\right|\geq\varepsilon\right)\\
&\leq \frac{18C}{\varepsilon^2}\left(\int_{0}^T\left(\int_{|x|\leq R}\|\si_t(x,\tilde{\mu}_t)-\si^{m}_t(x,\tilde{\mu}^{m}_t)\|_{HS}^{2p}\d x\right)^{\frac{q}{p}}\d t\right)^{\frac{1}{q}}\\
&+\frac{36dK}{\varepsilon^2}\int_{0}^T\tilde{\P}(|\tilde{X}_t|> R)\d t.
\end{align*}
Letting first $m\to\infty$ and then $R\to\infty$, we prove that when $n\to\infty$,
$$ \int_{0}^s\si^{n}_t(\tilde{X}^n_t,\L_{\tilde{X}^{n}_t})\d\tilde{W}^n_t\to\int_{0}^s\si_t(\tilde{X}_t,\L_{\tilde{X}_t})\,\d \tilde{W}_t$$
in probability $\tt\P$, uniformly in $s\in[0, T ]$.\end{proof}

\subsection{Proof of Theorem \ref{T1.1}(3)}
We will use the following result for the maximal operator:
\begin{align}\label{max}
\M h(x):=\sup_{r>0}\frac{1}{|B(x,r)|}\int_{B(x,r)}h(y)\d y,\ \  h\in L^1_{loc}(\mathbb{R}^d), x\in \R^d,
 \end{align}
 where $B(x,r):=\{y: |x-y|<r\},$  see   \cite[Appendix A]{CD}.

 \begin{lem} \label{Hardy}  There exists a constant $C>0$ such that for any continuous and weak differentiable function $f$,
 \beq\label{HH1}
|f(x)-f(y)|\leq C|x-y|(\M |\nabla f|(x)+\M |\nabla f|(y)),\ \  {\rm a.e.}\ x,y\in\R^d.\end{equation}
Moreover, for any $p>1$,    there exists a constant $C_{p}>0$ such that
\beq\label{HH2}
\|\M f\|_{L^p}\leq C_{p}\|f\|_{L^p},\ \ f\in L^p(\R^d).
 \end{equation}
\end{lem}

 Let $X$ and $Y$ be two solutions to \eqref{E1} with $X_0=Y_0$, and let $\mu_t=\L_{X_t}, \nu_t=\L_{Y_t}, t\in [0,T].$ Then $\mu_0=\nu_0$.
Let $$b_t^\mu(x)= b_t(x, \mu_t),\ \ \ \si_t^\mu(x)= \si_t(x,\mu_t),\ \ (t,x)\in [0,T]\times\R^d,$$ and define  $b_t^\nu, \si_t^\nu$ in the same way using $\nu_t$ replacing $\mu_t$. Then
\beq\label{E1'}\beg{split}
&\d X_t= b^\mu_t(X_t)\,\d t+ \sigma^\mu_t(X_t)\,\d W_t,\\
&\d Y_t= b_t^\nu(Y_t)\d t +\si_t^\nu(Y_t) \d W_t.\end{split}
\end{equation}
For any $\lambda>0$,    consider the following PDE for $u: [0,T]\times\R^d\to \R^d$:
\beq\label{PDE}
\frac{\partial u_t}{\partial t}+\frac{1}{2}\mathrm{Tr} (\sigma^\mu_t(\sigma_t^\mu)^\ast\nabla^2u_t)+\nabla_{b_t^\mu}u_t+b_t^\mu=\lambda u_t,\ \ u_T=0.
\end{equation} By Lemma \ref{KK} and  \cite[Theorem 3.1]{Z3},
when $\ll$ is large enough   \eqref{PDE} has a unique solution $\mathbf{u}^{\lambda,\mu}$ satisfying
\begin{align}\label{u0}
\|\nabla \mathbf{u}^{\lambda,\mu}\|_{\infty}\leq \frac{1}{5},
\end{align}
and \beq\label{u01} \|\nabla^2 \mathbf{u}^{\lambda,\mu}\|_{L^{2q}_{2p}(T)}<\infty.\end{equation}
Let $\theta^{\lambda,\mu}_t(x)=x+\mathbf{u}^{\lambda,\mu}_t(x)$. By \eqref{E1'}, \eqref{PDE} and It\^o's formula (see \cite{Z2} for more details), we have
\beq\label{E-X}
\d \theta^{\lambda,\mu}_t(X_t)= \lambda \mathbf{u}^{\lambda,\mu}_t(X_t)\d t+ (\nabla\theta_t^{\lambda,\mu}\sigma^\mu_t)(X_t)\,\d W_t,
\end{equation}
and
\beq\begin{split}\label{E-Y}
\d \theta^{\lambda,\mu}_t(Y_t)&=\lambda \mathbf{u}^{\lambda,\mu}_t(Y_t)\d t+(\nabla\theta_t^{\lambda,\mu}\sigma^\nu_t)(Y_t)\,\d W_t
+[\nabla\theta_t^{\lambda,\mu}(b^\nu_t-b^\mu_t)](Y_t)\d t\\
&+\frac{1}{2}\mathrm{Tr} [(\sigma^\nu_t(\sigma^\nu_t)^\ast-\sigma^\mu_t(\sigma^\mu_t)^\ast)\nabla^2\mathbf{u}^{\lambda,\mu}_t](Y_t)\d t.
\end{split}\end{equation}
  Let $\xi_t=\theta^{\lambda,\mu}_t(X_t)-\theta^{\lambda,\mu}_t(Y_t)$. By \eqref{E-X}, \eqref{E-Y} and It\^o's formula, we obtain
\begin{equation*}\begin{split}
\d|\xi_t|^2
=&2\lambda\left<\xi_t,\mathbf{u}^{\lambda,\mu}_t(X_t)-\mathbf{u}^{\lambda,\mu}_t(Y_t)\right\>\d t\\
&+2\left\<\xi_t,[(\nabla\theta_t^{\lambda,\mu}\sigma^\mu_t)(X_t)-(\nabla\theta_t^{\lambda,\mu}\sigma^\nu_t)(Y_t)]\d W_t\right\>\\
&+\left\|(\nabla\theta_t^{\lambda,\mu}\sigma^\mu_t)(X_t)-(\nabla\theta_t^{\lambda,\mu}\sigma^\nu_t)(Y_t)\right\|^2_{HS}\,\d t\\
&-2\left\<\xi_t, [\nabla\theta_t^{\lambda,\mu}(b^\nu_t-b^\mu_t)](Y_t)\right\>\d t\\
&-\left\<\xi_t,\mathrm{Tr} [(\sigma^\nu_t(\sigma^\nu_t)^\ast-\sigma^\mu_t(\sigma^\mu_t)^\ast)\nabla^2\mathbf{u}^{\lambda,\mu}_t](Y_t)\right\>\d t.
\end{split}\end{equation*}So, for any $m\ge 1$,
\beq\label{NN1}\beg{split}
\d|\xi_t|^{2m}
 =\, &2m\lambda|\xi_t|^{2(m-1)}\left<\xi_t,\mathbf{u}^{\lambda,\mu}_t(X_t)-\mathbf{u}^{\lambda,\mu}_t(Y_t)\right\>\d t\\
&+2m|\xi_t|^{2(m-1)}\left\<\xi_t,[(\nabla\theta_t^{\lambda,\mu}\sigma^\mu_t)(X_t)-(\nabla\theta_t^{\lambda,\mu}\sigma^\nu_t)(Y_t)]\d W_t\right\>\\
&+m|\xi_t|^{2(m-1)}\left\|(\nabla\theta_t^{\lambda,\mu}\sigma^\mu_t)(X_t)-(\nabla\theta_t^{\lambda,\mu}\sigma^\nu_t)(Y_t)\right\|^2_{HS}\,\d t\\
&+2m(m-1) |\xi_t|^{2(m-2)}\left|[(\nabla\theta_t^{\lambda,\mu}\sigma^\mu_t)(X_t)-(\nabla\theta_t^{\lambda,\mu}\sigma^\nu_t)(Y_t)]^\ast\xi_t \right|^2\d t\\
&-2m|\xi_t|^{2(m-1)}\left\<\xi_t, [\nabla\theta_t^{\lambda,\mu}(b^\nu_t-b^\mu_t)](Y_t)\right\>\d t\\
&-m|\xi_t|^{2(m-1)}\left\<\xi_t,\mathrm{Tr} [(\sigma^\nu_t(\sigma^\nu_t)^\ast-\sigma^\mu_t(\sigma^\mu_t)^\ast)\nabla^2\mathbf{u}^{\lambda,\mu}_t](Y_t)\right\>\d t.\end{split}\end{equation}
By \eqref{u0}, \eqref{LIP}, Lemma \ref{Hardy}, and noting that the distributions of $X_t$ and $Y_t$ are   absolutely continuous with respect to the Lebesgue measure, we may find out a constant $c_1>0$ such that
\beq\label{XPP1}|\xi_t|^{2(m-1)} |\xi_t|\cdot|\mathbf{u}^{\lambda,\mu}_t(X_t)-\mathbf{u}^{\lambda,\mu}_t(Y_t)|\le c_1 |\xi_t|^{2m},\end{equation}
\beq\label{XPP2}\beg{split} &|\xi_t|^{2(m-2)} \left|[(\nabla\theta_t^{\lambda,\mu}\sigma^\mu_t)(X_t)-(\nabla\theta_t^{\lambda,\mu}\sigma^\nu_t)(Y_t)]^\ast\xi_t \right|^2\\
  &\le |\xi_t|^{2(m-1)}\left\|(\nabla\theta_t^{\lambda,\mu}\sigma^\mu_t)(X_t)-(\nabla\theta_t^{\lambda,\mu}\sigma^\nu_t)(Y_t)\right\|^2_{HS}\\
&\le |\xi_t|^{2(m-1)} \Big\{C |\xi_t| \scr M\big(\|\nn^2\theta_t^{\ll,\mu}\|+\|\nabla\si_t^\mu\|\big)(X_t)\\
 &\qquad\qquad\qquad + C |\xi_t| \scr M\big(\|\nn^2\theta_t^{\ll,\mu}\|+\|\nabla\si_t^\mu\|\big)(Y_t)+ \W_\theta(\mu_t,\nu_t)\Big\}^2\\
&\le c_1|\xi_t|^{2m} \big\{\scr M\big(\|\nn^2\theta_t^{\ll,\mu}\|+\|\nabla\si_t^\mu\|\big)(X_t) +   \scr M\big(\|\nn^2\theta_t^{\ll,\mu}\|+\|\nabla\si_t^\mu\|\big)(Y_t)\big\}^2\\
&\quad + c_1 |\xi_t|^{2m}+c_1\W_\theta(\mu_t,\nu_t)^{2m},\end{split}\end{equation}
\beq\label{XPP3} \beg{split}&|\xi_t|^{2(m-1)} |\xi_t| \cdot |\{\nn\theta_t^{\ll,\mu}(b_t^\nu-b_t^\mu)\}(Y_t)|\\
&\le L \|\nn\theta^{\ll,\mu}\|_{T,\infty }|\xi_t|^{2(m-1)}|\xi_t| \W_\theta(\mu_t,\nu_t)\le c_1\big(|\xi_t|^{2m} +\W_\theta(\mu_t,\nu_t)^{2m}\big), \end{split}\end{equation} and
  for some constants $c_0,c_1>0$
\beq\label{XPP4}\beg{split} &|\xi_t|^{2(m-1)}|\xi_t|\cdot \big|\mathrm{Tr} [(\sigma^\nu_t(\sigma^\nu_t)^\ast-\sigma^\mu_t(\sigma^\mu_t)^\ast)\nabla^2\mathbf{u}^{\lambda,\mu}_t](Y_t)\big|\\
&\le c_0 |\xi_t|^{2m-1} \W_\theta(\mu_t,\nu_t) \|\nn^2\mathbf{u}^{\lambda,\mu}_t\|(Y_t)\\
&\le c_1 |\xi_t|^{2m}|\|\nn^2\mathbf{u}^{\lambda,\mu}_t\|^{\ff{2m}{2m-1}}(Y_t) + c_1\W_\theta(\mu_t,\nu_t)^{2m}.\end{split}\end{equation}
Combining \eqref{XPP1}-\eqref{XPP4} with \eqref{NN1}, and noting that $\ff{2m}{2m-1}\le 2$, we arrive at
\beq\label{NNP}\d |\xi_t|^{2m} \le c_2 |\xi_t|^{2m} \d A_t + c_2 \W_\theta(\mu_t,\nu_t)^{2m}\d t + \d M_t\end{equation}
for some constant $c_2>0$, a local martingale $M_t$,  and
\begin{align*}
A_t:=\int_0^t\Big\{&1+ |\nn^2{\mathbf u}_s^{\ll,\mu}(Y_s)|^2 +\big(\scr M\big(\|\nn^2\theta_s^{\ll,\mu}\|+\|\nn\si_s^\mu\|\big)(X_s)  \\
&+  \scr M\big(\|\nn^2\theta_s^{\ll,\mu}\|+\|\nn\si_s^\mu\|\big)(Y_s)\big)^2\Big\}\d s.
\end{align*}
By the stochastic Gronwall lemma due to \cite[Lemma 3.8]{XZ}, when $2m>\theta$ this implies
\beq\label{NN2}\W_\theta(\mu_t,\nu_t)^{2m}\le (\E |\xi_t|^\theta)^{\ff{2m}\theta} \le c_2\big(\E\e^{\ff{c_2\theta}{2m-\theta}A_t}\big)^{\ff{2m-\theta}{\theta}} \int_0^t\W_\theta(\mu_s,\nu_s)^{2m}\d s,\ \ t\in [0,T].\end{equation}
Since by Lemma \ref{KK}, \eqref{HH2}, \eqref{u01} and the Khasminskii type estimate, see for instance \cite[Lemma 3.5]{XZ}, we have
$$\E\e^{\ff{c_2\theta}{2m-\theta}A_T}<\infty,$$
so that by Gronwall's lemma we prove $\W_\theta(\mu_t,\nu_t)=0$ for all $t\in [0,T].$ Then by \eqref{E1'} both $X_t$ and $Y_t$ solve the same SDE with coefficients $b_t^\mu$ and $\si_t^\mu$, and due to \cite{Z2}, the condition  $1_D(|b_t^\mu|^2+|\nn \si_t^\mu|^2)\in L_p^q(T)$ for compact $D\subset \R^d$ implies the pathwise uniqueness of this SDE, so we conclude that $X_t=Y_t$ for all $t\in [0,T].$

\subsection{Proof of Corollary \ref{C1.2} and Corollary \ref{C1.3}}
 \begin{proof}[Proof of Corollary \ref{C1.2}]
 We set $a_t(x,\mu) :=(\si\si^\ast)_t(x,\mu)$ for $t \in [0,T]$, and $b_t( x,\mu) := 0$, $a_t( x,\mu) :=  I$ for $t \in \R\setminus [0,T]$.
Let $0\le \rr\in C_0^\infty(\R\times\R^d)$ with support contained in $\{(r,x): |(r,x)|\le 1\}$ such that $\int_{\R\times\R^d} \rr(r,x)\d r\d x=1.$
For any $n\ge 1$, let $\rr_n(r,x)= n^{d+1} \rr(nr, nx)$ and define
\begin{equation}\begin{split}\label{approx}
&a^n_t(x,\mu)=\int_{\R\times\R^d} \sigma_s\sigma^\ast_s(x',\mu) \rr_n (t-s, x-x')\d s \d x',\\
&b^n_t(x,\mu)=\int_{\R\times\R^d} b_s(x',\mu) \rr_n (t-s, x-x')\d s \d x',\ \ (t,x,\mu)\in \R\times\R^d\times\scr P.
\end{split}\end{equation}
 Let $\hat{\sigma}_t^n=\ss{a^n_t}$ and $\hat{\sigma}_t=\ss{a_t}$. Consider the following SDE:
\beq\label{E1''} \d X_t= b_t(X_t, \L_{X_t})\d t +\hat{\si}_t(X_t, \L_{X_t})\d W_t.\end{equation}
We first show that  $(b,\hat{\si})$ satisfies assumption $(H^\theta)$. Firstly, \eqref{EX1}-\eqref{EX2} and the continuity in the third variable of $B$ and $\Sigma$ imply that $b$ and $\sigma$ are continuous in the third variable $\mu\in \scr P_\theta$. Thus, (1) in $(H^\theta)$ holds. As to $(H^\theta)$ (2),
since by \cite {Z2}, it holds that $$\lim_{n\to\infty}\|F-F\ast\rho_n\|_{L^q_p(T)}=0,$$ there exists a subsequence $n_k$ such that
 $$\|F-F\ast\rho_{n_k}\|_{L^q_p(T)}<2^{-k}.$$ Letting $$G=\sum_{k=1}^{\infty}|F-F\ast\rho_{n_k}|+F,$$ then $\|G\|_{L^q_p(T)}\leq 1+\|F\|_{L^q_p(T)}$ and noting $|b^{n_k}|^2\leq K+F\ast\rho_{n_k}$, we have $|b^{n_k}|^2\leq K+G$. So, using the subsequence $b^{n_k}$ replacing $b^n$,   we verify condition (2) in $(H^\theta)$.
Finally, by \eqref{EX1}, for any $n\ge 1$ there exists a constant  $c_n>0$ such that
\begin{align*}   |b_t^n(x,\mu)-b_s^n(x',\nu)|+ \|\hat{\sigma}_t^n(x,\mu)-\hat{\sigma}_s^n(x',\nu)\|
 \le c_n \big(|t-s|+|x-x'| + \W_1(\mu,\nu)\big)
\end{align*}
holds for all $s,t\in \R, x,x'\in \R^d$ and $\mu,\nu\in \scr P_1$. So, for any $\theta\ge 1,$ condition (3) in $(H^\theta)$ holds. By Theorem \ref{T1.1} (1), SDE \eqref{E1''} has a weak solution.  Noting that $\sigma\sigma^\ast=\hat{\sigma}\hat{\sigma}^\ast$, the SDE \eqref{E1} also has a weak solution.
Finally, the strong existence and uniqueness   follow from Theorem \ref{T1.1} (2) and (3).
\end{proof}
 \begin{proof}[Proof of Corollary \ref{C1.3}]  Let $b_t^n$ and $a_t^n$ be in \eqref{approx}, and let $\hat{\sigma}_t^n=\ss{a^n_t}$ and $\hat{\sigma}_t=\ss{a_t}$. Then \eqref{LIP} and \eqref{approx} imply $(b,\hat{\sigma})$ satisfy $H^\theta$. Then we may complete the proof  as in the proof of Corollary \ref{C1.3} (1).
 \end{proof}
 \section{Proofs of Theorems \ref{T3.1} and \ref{T5.1}}

 \subsection{Proof of Theorem \ref{T3.1}}
 According to \cite[Theorem 1.2 (2)]{WZ'} for $d_1=0$,   Corollary \ref{C1.3}, and Lemma \ref{SS},     {\bf(H)} implies the existence and uniqueness of solution to \eqref{E1}.  For any $\mu\in \scr P_2$ we let $\mu_t=P_t^*\mu$ be the distribution of $X_t$
which solves \eqref{E11} with $\L_{X_0}=\mu.$

We first figure out the outline of proof using coupling by change of measure as in \cite{W11,Wbook}.  From now on,
we fix $t_0\in (0,T]$ and   $\mu_0,\nu_0\in \scr P_2$, and take  $\F_0$-measurable variables $X_0$ and $Y_0$ in $\R^d$ such that $\L_{X_0}=\mu_0, \L_{Y_0}=\nu_0$ and
\beq\label{I1} \E |X_0-Y_0|^2 = \W_2(\mu_0,\nu_0)^2.\end{equation}
Let $X_t$  with $\L_{X_0}=\mu_0$ solve \eqref{E11}, we have
\beq\label{CP1} \d X_t= b_t(X_t,\mu_t)\d t+\si_t(X_t)\d W_t.\end{equation}
To establish the log-Harnack inequality, We construct  a process  $Y_t$ such that for a weighted probability measure $\Q:=R\P$
\beq\label{CP2} X_{t_0}=Y_{t_0}\ \Q\text{-a.s., \ \ and}\ \L_{Y_{t_0}}|\Q=P_{t_0}^*\nu_0=:\nu_{t_0}. \end{equation}
Then
$$ (P_{t_0} f)(\nu_0)= \E_\Q[f(Y_{t_0})]=\E[R_{t_0}f(X_{t_0})],\ \ f\in \B_b(\R^d).$$ So, by Young's inequality we obtain the log-Harnack inequality:
\beq\label{LHI}\beg{split} (P_{t_0}\log f)(\nu_0)&\le \E[R_{t_0}\log R_{t_0}]+ \log\E[f(X_{t_0})]\\
&=\log (P_{t_0}f)(\mu_0)+ \E[R_{t_0}\log R_{t_0}],\ \ f\in \B_b^+(\R^d), f\geq 1.\end{split}\end{equation}

To construct the desired $Y_t$, we follow the line of \cite{WZ'} using Zvonkin's transform.
As shown in \cite[Theorem 3.10]{WZ'} for $d_1=0$ that Assumption {\bf(H)}  implies that for large enough $\ll>0$, the PDE \eqref{PDE} has a unique solution $\mathbf{u}^{\lambda,\mu}$ satisfying
\begin{align}\label{u}
\|\mathbf{u}^{\lambda,\mu}\|_{\infty}+\|\nabla \mathbf{u}^{\lambda,\mu}\|_{\infty}+\|\nabla^2 \mathbf{u}^{\lambda,\mu}\|_{\infty}\leq \frac{1}{5}.
\end{align} $\|\nn^2\mathbf{u}^{\lambda,\mu}\|_{\infty}<\infty$ together with the Lipschitzian continuity of $\sigma$ implies that the increasing process $A_t$ in \eqref{NNP} satisfies
$$\d A_t\le c \d t$$ for some constant $c>0$.
Moreover, $\E |\xi_t|^2\ge c'\W_2(\mu_t,\nu_t)^2$ holds for some constant $c'>0$. So, with $m=1, \theta=2, \L_{X_0}=\mu_0$ and $\L_{Y_0}=\nu_0$, the inequality \eqref{NNP} gives
\beq\label{*D*} \W_2(\mu_t,\nu_t)\le \kk \W_2(\mu_0,\nu_0),\ \ t\in [0,T]\end{equation} for some constant $\kk>0$.

As in \cite[\S2]{W11},   let $\gamma=\frac{72}{25}K+\frac{2d}{25\delta}+\frac{12\lambda}{25}$ and take
\beq\label{Xi} \zeta_t= \ff {12} {25\gamma} \Big(1-\e^{\frac{25\gamma}{16}(t-t_0)}\Big),\ \ t\in [0,t_0],\end{equation}
and let   $Y_t$ solve   the modified  SDE
\beq\label{CY} \d Y_t = \Big\{b_t(Y_t,\nu_t)+ \ff 1 {\zeta_t} \si_t(Y_t)\si_t(X_t)^{-1}(X_t-Y_t)\Big\}\d t+ \si_t(Y_t) \d W_t,\ \ t\in [0,t_0).\end{equation}
Since  $\sup_{t\in [0,T]}\nu_t(|\cdot|^2)<\infty$,   this SDE has a unique solution $(Y_t)_{t\in [0,t_0)}$. Let
$$\tau_n:=t_0\land  \inf\{t\in [0,t_0): |X_t|+|Y_t|\ge n\},\ \ n\ge 1,$$  where $\inf\emptyset :=\infty$ by convention.
We have $\tau_n\uparrow t_0$ as $n\uparrow\infty$.
To see that the process $Y$ meets the above requirement,   we first prove that
\beq\label{R} R_s:= \exp\bigg[\int_0^s \ff 1 {\zeta_t} \big\<\si_t(X_t)^{-1}(Y_t-X_t),\d W_t\big\>-\ff 1 2 \int_0^s \ff {|\si_t(X_t)^{-1}(Y_t-X_t)|^2} {\zeta_t^2}\d t\bigg]\end{equation} for $s\in [0,t_0)$ is a uniformly integrable martingale, and hence extends also to time $t_0$.

\beg{lem}\label{L3.1} Assume {\bf (A1)}-{\bf (A2)} and let $X_0,Y_0$ be two $\F_0$-measurable random variables such that $\L_{X_0}=\mu_0, \L_{Y_0}=\nu_0$, and    \beq\label{I1'} \E|X_0-Y_0|^2= \W_2(\mu_0,\nu_0)^2.\end{equation}  Then there exists a constant $c>0$  uniformly in $t_0\in (0,T)$ such that
\beq\label{ES} \sup_{t\in [0,t_0)}\E[R_t\log R_t]\le \ff c {t_0}\W_2(\mu_0,\nu_0)^2.\end{equation} Consequently,
$R_t$ extends to $t=t_0$, $\Q:= R_{t_0}\P$ is a probability measure under which $\eqref{CY}$ has a unique solution $(Y_t)_{t\in [0,t_0]}$ satisfying
\beq\label{ES'0} \Q(X_{t_0}=Y_{t_0})=1. \end{equation}
\end{lem}
\beg{proof} By {\bf (A1)}, for any $n\ge 1$ and $t\in (0,t_0)$, the process $(R_{s\land \tau_n})_{s\in [0,t]}$ is a uniformly integrable continuous martingale. So, for the first assertion it suffices to find out a constant $c>0$ uniformly in $t_0\in (0,T)$  such that
\beq\label{ES'} \sup_{n\ge 1}\E[R_{t\land\tau_n}\log R_{t\land\tau_n}]\le \ff c {t_0}  \W_2(\mu_0,\nu_0)^2,\ \ t\in [0,t_0).\end{equation}
To this end, for fixed $t\in (0,T)$ and $n\ge 1$, we consider the weighted  probability $\Q_{t,n}:= R_{t\land \tau_n}\P$.   By Girsnaov's theorem
$(\tt W_s)_{s\in [0, t\land\tau_n]}$ is a $d$-dimensional Brownian motion under $\Q_{t,n}$. Reformulating \eqref{CP1} and \eqref{CY} as
\beg{equation*}\beg{split} &\d X_s= b_s(X_s,\mu_s) - \ff{X_s-Y_s}{\zeta_s}\d s +\si_s(X_s)\d \tt W_s,\\
&\d Y_s=  b_s(Y_s,\nu_s)+   \si_s(Y_s) \d \tt W_s,\ \ s\in [0, t\land \tau_n],\end{split}\end{equation*}
where
$$\tt W_t=W_t+\int_{0}^{t}\ff 1 {\zeta_s} \si_s(X_s)^{-1}(X_s-Y_s)\d W_s.$$
Next, we fix $\lambda=\lambda_0$. Letting $\theta^{\lambda,\mu}_t(x)=x+\mathbf{u}^{\lambda,\mu}_t(x)$, combining \eqref{PDE} and It\^{o}'s formula, we arrive at
\beq\label{E-X'}
\d \theta^{\lambda,\mu}_t(X_t)= \lambda \mathbf{u}^{\lambda,\mu}_t(X_t)\d t+ (\nabla\theta_t^{\lambda,\mu}\sigma_t)(X_t)\,\d \tilde{W}_t-\nabla\theta_t^{\lambda,\mu}(X_t)\ff{X_t-Y_t}{\zeta_t}\d t,
\end{equation}
and
\beq\begin{split}\label{E-Y'}
\d \theta^{\lambda,\mu}_t(Y_t)&=\lambda \mathbf{u}^{\lambda,\mu}_t(Y_t)\d t+(\nabla\theta_t^{\lambda,\mu}\sigma_t)(Y_t)\,\d \tilde{W}_t+[\nabla\theta_t^{\lambda,\mu}(b^\nu_t-b^\mu_t)](Y_t)\d t
\end{split}\end{equation}
By  It\^o's formula under probability $\Q_{t,n}$, we obtain
\beq\begin{split}\label{EX-Y'}
&\d |\theta^{\lambda,\mu}_t(Y_t)-\theta^{\lambda,\mu}_t(X_t)|^2\\
&=2\langle\theta^{\lambda,\mu}_t(X_t)-\theta^{\lambda,\mu}_t(Y_t),\lambda \mathbf{u}^{\lambda,\mu}_t(X_t)-\lambda \mathbf{u}^{\lambda,\mu}_t(Y_t)\rangle\d t\\
&+2\langle\theta^{\lambda,\mu}_t(X_t)-\theta^{\lambda,\mu}_t(Y_t),(\nabla\theta_t^{\lambda,\mu}\sigma_t)(X_t)\d \tilde{W}_t-(\nabla\theta_t^{\lambda,\mu}\sigma_t)(Y_t)\d \tilde{W}_t\rangle\\
&+\|\nabla\theta_t^{\lambda,\mu}\sigma_t)(X_t)-\nabla\theta_t^{\lambda,\mu}\sigma_t)(Y_t)\|^2_{HS}\d t\\
&-2\langle\theta^{\lambda,\mu}_t(X_t)-\theta^{\lambda,\mu}_t(Y_t),[\nabla\theta_t^{\lambda,\mu}(b^\nu_t-b^\mu_t)](Y_t)\d t\rangle\\
&-2\Big\langle\theta^{\lambda,\mu}_t(X_t)-\theta^{\lambda,\mu}_t(Y_t),\nabla\theta_t^{\lambda,\mu}(X_t)\ff{X_t-Y_t}{\zeta_t}\d t\Big\rangle.
\end{split}\end{equation}
By \eqref{u} we have
\begin{align*}
&-\Big\langle\theta^{\lambda,\mu}_t(X_t)-\theta^{\lambda,\mu}_t(Y_t),\nabla\theta_t^{\lambda,\mu}(X_t)\ff{X_t-Y_t}{\zeta_t}\Big\rangle\\
&=-\Big\langle X_t-Y_t +\mathbf{u}^{\lambda,\mu}_t(X_t)-\mathbf{u}^{\lambda,\mu}_t(Y_t),\ff{X_t-Y_t}{\zeta_t}+\nabla \mathbf{u}_t^{\lambda,\mu}(X_t)\ff{X_t-Y_t}{\zeta_t}\Big\rangle\\
&=-\Big\langle X_t-Y_t,\ff{X_t-Y_t}{\zeta_t}\Big\rangle-\Big\langle \mathbf{u}^{\lambda,\mu}_t(X_t)-\mathbf{u}^{\lambda,\mu}_t(Y_t),\ff{X_t-Y_t}{\zeta_t}\Big\rangle\\
&-\Big\langle X_t-Y_t,\nabla \mathbf{u}_t^{\lambda,\mu}(X_t)\ff{X_t-Y_t}{\zeta_t}\Big\rangle-\Big\langle \mathbf{u}^{\lambda,\mu}_t(X_t)-\mathbf{u}^{\lambda,\mu}_t(Y_t),\nabla \mathbf{u}_t^{\lambda,\mu}(X_t)\ff{X_t-Y_t}{\zeta_t}\Big\rangle\\
&\leq -\frac{14}{25}\ff{|X_t-Y_t|^2}{\zeta_t}.
\end{align*} So,
\begin{align*} \d  |\theta^{\lambda,\mu}_s(Y_s)-\theta^{\lambda,\mu}_s(X_s)|^2&\le \Big\{\gamma|X_s-Y_s|^2 +\frac{72}{25}\kk_2(T) |X_s-Y_s|\W_2(\mu_s,\nu_s) -\frac{4}{5}\ff{|X_s-Y_s|^2}{\zeta_s}\Big\}\d s \\
&+\d M_s,\ \ s\in [0, t\land \tau_n]
\end{align*}
for some $\Q_{t,n}$-martingale $M_s$. By  \eqref{Xi} we have
$$\frac{4}{5}-\gamma\zeta_s+ \frac{16}{25}\zeta_s' =\ff {8} {25},$$
 By It\^o's formula, there exists a constant $c_2>0$ such that
Then
\beq\label{V1}\beg{split}&\d \ff{|\theta^{\lambda,\mu}_s(Y_s)-\theta^{\lambda,\mu}_s(X_s)|^2}{\zeta_s}\\
&\le \ff{\d M_s}{\zeta_s}+ c_2\W_2(\mu_s,\nu_s)^2 \d s
 -\ff{|X_s-Y_s|^2}{\zeta_s^2}\Big\{\frac{4}{5}-\gamma\zeta_s+ \frac{16}{25}\zeta_s'- \ff {1} {25}\Big\} \d s\\
 &\le \ff{\d M_s}{\zeta_s}+ c_2\W_2(\mu_s,\nu_s)^2 \d s
 -\ff{7|X_s-Y_s|^2}{25\zeta_s^2},\ \ s\in [0, t\land \tau_n].\end{split}\end{equation}
Combining   this with \eqref{*D*}, \eqref{I1} and \eqref{V1}, we arrive at
\beq\label{V2}  \E_{\Q_{t,n}} \int_0^{t\land\tau_n} \ff{|X_s-Y_s|^2}{\zeta_s^2}\d s\le \ff {c_1} {t_0}\W_2(\mu_0,\nu_0)^2,\ \ t\in[0,t_0)  \end{equation} for some constant $c_1>0$. Therefore, there exists a constant $C>0$ such that
\beg{align*} \E[R_{t\land \tau_n}\log R_{t\land \tau_n}]&= \ff 1 2 \E_{\Q_{t,n}}\int_0^{t\land\tau_n}  \ff {|\si_s(X_s)^{-1}(Y_s-X_s)|^2} {\zeta_s^2}\d s \\
&\le \ff{C}{t_0}\W_2(\mu_0,\nu_0)^2,\ \ t\in (0,t_0).\end{align*}  Thus, \eqref{ES} holds.

By \eqref{ES} and the martingale convergence theorem, $(R_{t})_{t\in [0,t_0]}$ is a uniformly integrable martingale, so $\Q:= R_{t_0}\P$ is a probability measure. By Girsanov theorem, we can reformulate \eqref{CY} as
\beq\label{ET}\d Y_t= b_t(Y_t,\nu_t)\d t +\si_t(Y_t)\d \tt W_t,\end{equation}
which has a unique solution $(Y_t)_{t\in [0,t_0]}$. By \eqref{ES},
$$\E_\Q \int_0^{t_0} \ff{|X_t-Y_t|^2}{\zeta_t^2}\d t<\infty.$$ Since $X_t-Y_t$ is continuous and $\int_0^{t_0}\ff 1 {\zeta_t}\d t=\infty$, this implies
$\Q(X_{t_0}=Y_{t_0})=1.$
 \end{proof}
\beg{proof}[Proof of Theorem \ref{T3.1}]    Consider the distribution dependent SDE
$$\d\tt X_t= b_t(\tt X_t, \scr L_{\tt X_t}|\tt \P)\d t+ \si_t(\tt X_t)\d \tt W_t,\ \  \tt X_0= Y_0.$$ By the weak uniqueness we  have
$\scr L_{\tt X_t}|\tt \P= P_t^*\nu_0= \nu_t$ for $t\in [0,t_0]$. Combining this with   \eqref{ET} and the strong uniqueness, we conclude that
$\tt X_t=Y_t$ for $t\in [0,T]$. Therefore,  \eqref{LHI} and Lemma \ref{L3.1} lead to
\beg{align*}  (P_{t_0} \log f)(\nu_0)\le \log (P_{t_0}f)(\mu_0) +\ff C {t_0}\W_2(\mu_0,\nu_0)^2,\ \ t_0\in (0,T].\end{align*}
Finally, the Harnack inequality with power \eqref{H2'} follows from \cite[Section 3.4]{Wbook}.
\end{proof}

\subsection{Proof of Theorem \ref{T5.1}}

\beg{proof}Fix $t_0>0$.
Denote $\mu_t=P_t^*\mu_0=\L_{X_t}, t\in [0,t_0]$. Then \eqref{E5} becomes
\beq\label{E5'}  \d X_t= b_t(X_t, \mu_t)\d t +\si_t(\mu_t) \d W_t,\ \ \L_{X_0}=\mu_0.\end{equation}
Let $Y_t=X_t+\ff {tv}{t_0},\ t\in [0,t_0]$. Then
$$ \d Y_t= b_t(Y_t, \mu_t)\d t +\si_t(\mu_t) \d \tt W_t,\ \ \L_{Y_0}=\mu_0, t\in [0,t_0], $$ where
\beg{align*} &\tt W_t:= W_t +\int_0^t \eta_s\d s,\\
&\eta_t:= \si_t^{-1}\Big\{\ff v {t_0} + b_t(X_t,\mu_t)-b_t\Big(X_t+\ff {tv}{t_0}, \mu_t\Big)\Big\}.\end{align*}
Let $R_{t_0}= \exp[-\int_0^{t_0} \<\eta_t, \d W_t\>-\ff 1 2\int_0^{t_0} |\eta_s|^2\d s].$ By the Girsanov theorem we obtain
$$(P_{t_0} f)(\mu_0)=\E [R_{t_0} f(Y_{t_0})]= \E[R_{t_0} f(X_{t_0}+v)]\le (P_{t_0} f^p(v+\cdot))^{\ff 1 p}(\mu_0) \big(\E R_{t_0}^{\ff p{p-1}}\big)^{\ff {p-1}p},$$
and by Young's inequality, we obtain
\begin{align*}(P_{t_0} \log f)(\mu_0)&=\E [R_{t_0} \log f(Y_{t_0})]\\
&= \E[R_{t_0} \log f(X_{t_0}+v)]\le \log P_{t_0} f(v+\cdot)(\mu_0)+ \E R_{t_0}\log R_{t_0}.\end{align*}
Then we have
\beg{align*} &\E R_{t_0}^{\ff p {p-1}}
 \le \sup_{\Omega}\e^{\ff{p}{2(p-1)^2} \int_0^{t_0} |\eta_s|^2\d s} \\
&\le \exp\bigg[\ff{p\, \int_0^{t_0} \|\si_t^{-1}\|_{\infty}^2 \big\{|v|/{t_0}+\phi(t|v|/{t_0})\big\}^2\d t}{2(p-1)^2}\bigg].\end{align*}
and
\beg{align*} &\E R_{t_0}\log R_{t_0}= \E_\Q \log R_{t_0}
 \le \frac{1}{2}\E_\Q\int_0^{t_0} |\eta_s|^2\d s\\
&\le \frac{1}{2}\int_0^{t_0} \|\si_t^{-1}\|_{\infty}^2 \big\{|v|/{t_0}+\phi(t|v|/{t_0})\big\}^2\d t.\end{align*}
\end{proof}

\end{document}